\documentclass[10pt]{amsart}
\usepackage{amsmath,amssymb,amsthm,amsfonts,graphicx,color}
\usepackage{amssymb}
\usepackage{amsfonts,amsthm,amsmath,latexsym,bm,graphics}

\makeatother
\numberwithin{equation}{section}

\newtheorem{lemma}{Lemma}[section]

\newtheorem{teo}[lemma]{Theorem}
\newtheorem{proposition}[lemma]{Proposition}
\newtheorem{corollary}[lemma]{Corollary}
\theoremstyle{remark}
\newtheorem{remark}[lemma]{Remark}

\newcommand{\ve} {\varepsilon}
\newcommand{\R}{{\mathbb R}}

\newcommand{\be}{\begin{equation}}
\newcommand{\ben}{\begin{equation*}}
\newcommand{\ee}{\end{equation}}
\newcommand{\een}{\end{equation*}}
\newcommand{\BL}{\begin{lemma}}
\newcommand{\EL}{\end{lemma}}
\newcommand{\BT}{\begin{theorem}}
\newcommand{\ET}{\end{theorem}}
\newcommand{\BP}{\begin{proposition}}
\newcommand{\EP}{\end{proposition}}
\newcommand{\BC}{\begin{corollary}}
\newcommand{\EC}{\end{corollary}}
\def\bs{\begin{split}}
\def\es{\end{split}}

\DeclareMathOperator{\divergence}{div}

\DeclareMathOperator{\Real}{Re}

\DeclareMathOperator{\Ker}{Ker}
\DeclareMathOperator{\rango}{Rg}

\newcommand{\COMMENT}[1]{}
\newcommand{\eps}{\varepsilon}

\newcommand{\p}{\partial}
\newcommand{\pnu}[1]{\frac{\partial{#1}}{\partial\nu}}

\newcommand{\Ds}{(-\Delta)^s}

\newcommand{\loc}{\textnormal{loc}}
\newcommand{\norm}[2][]{\left\|{#2}\right\|_{#1}}

\newcommand{\set}[1]{\left\{#1\right\}}
\newcommand{\textin}{\text{ in }}

\newcommand{\textand}{\text{ and }}

\vbadness=\maxdimen

\begin{document}

\title[Supercritical NLS]{Bound state solutions for the supercritical fractional Schr\"odinger equation}

\author[W. Ao]{Weiwei Ao}

\address{Weiwei Ao
\hfill\break\indent
Wuhan University
\hfill\break\indent
Department of Mathematics and Statistics, Wuhan, 430072, PR China}
\email{wwao@whu.edu.cn}

\author[H. Chan]{Hardy Chan}

\address{Hardy Chan
\hfill\break\indent
University of British Columbia,
\hfill\break\indent
Department of Mathematics, Vancouver, BC V6T1Z2, Canada}
\email{hardy@math.ubc.ca}

\author[M.d.M. Gonz\'alez]{Mar\'ia del Mar Gonz\'alez}

\address{Mar\'ia del Mar Gonz\'alez
\hfill\break\indent
Universidad Aut\'onoma de Madrid
\hfill\break\indent
Departamento de Matem\'aticas, Campus de Cantoblanco, 28049 Madrid, Spain}
\email{mariamar.gonzalezn@uam.es}

\author[J. Wei]{Juncheng Wei}
\address{Juncheng. Wei
\hfill\break\indent
University of British Columbia
\hfill\break\indent
 Department of Mathematics, Vancouver, BC V6T1Z2, Canada} \email{jcwei@math.ubc.ca}

\begin{abstract}
We prove the existence of positive solutions for the supercritical nonlinear fractional Schr\"odinger equation $(-\Delta)^s u+V(x)u-u^p=0 \mbox{ in } \R^n$, with $u(x)\to 0$ as $|x|\to +\infty$, where $p>\frac{n+2s}{n-2s}$ for $s\in (0,1), \ n>2s$.  We show that if $V(x)=o(|x|^{-2s})$ as $|x|\to +\infty$, then for $p>\frac{n+2s-1}{n-2s-1}$, this problem admits a continuum of solutions.  More generally, for $p>\frac{n+2s}{n-2s}$, conditions for solvability are also provided. This result is the extension of the work by Davila, Del Pino, Musso and Wei to the fractional case.  Our main contributions are: the existence of a smooth, radially symmetric, entire solution of $(-\Delta)^s w=w^p \mbox{ in }\R^n$, and the analysis of its properties. The difficulty here is the lack of phase-plane analysis for a nonlocal ODE; instead we use conformal geometry methods together with Schaaf's argument as in the paper by Ao, Chan, DelaTorre, Fontelos, Gonz\'alez and Wei on the singular fractional Yamabe problem.
\end{abstract}

\date{}\maketitle
%\keywords{D}

%\abbreviations{}

\centerline{AMS subject classification:  35J61, 35R11, 53A30}

\section{Introduction}

Fix $s\in(0,1)$ and $n>2s$. We consider the following problem
\begin{equation}\label{problem-1}
\begin{cases}
(-\Delta)^s u+Vu-u^p=0 &\mbox{ in }\R^n,\\
\lim\limits_{|x|\to \infty}u(x)=0,\quad u>0,
\end{cases}
\end{equation}
where $V$ is a non-negative potential, for a supercritical power  nonlinearity, i.e., $p>\frac{n+2s}{n-2s}$.

Problem \eqref{problem-1} arises when considering standing wave solutions
for the nonlinear fractional Schr\"odinger equation
\begin{equation}\label{NLS}
-i\frac{\partial \psi}{\partial t}=(-\Delta)^s \psi -Q(y)\psi+|\psi|^{p-1}\psi,\end{equation}
that is, solutions of the form $\psi(t,y)=\exp{(i\lambda t)}u(y).$
If $u(y)$ is positive and vanishes at infinity, then $\psi$ satisfies \eqref{NLS} if and only if $u$ solves  \eqref{problem-1}.

The fractional Schr\"odinger equation, a fundamental tool in fractional quantum mechanics, was introduced by Laskin \cite{Laskin1,Laskin2,Laskin3} (see also the appendix in \cite{Davila-DelPino-Dipierro-Valdinoci}) as a result of extending the Feynman path integral from the Brownian-like paths, which yields the standard Sch\"rodinger equation, to L\'evy-like quantum mechanical paths.

There is by now a huge literature on the fractional
 Schr\"odinger equation, see for instance, \cite{Felmer-Quaas-Tan,Secchi,Cheng,Davila-DelPino-Dipierro-Valdinoci,Molica-Radulescu} for a subcritical power, while for the critical case we have \cite{Frank-Lenzmann,Ambrosio-Figueiredo,Guo-He}, but this list is by no means complete. However, none of these deals with the supercritical regime. The present paper is one of the first attempts in this direction.

Our main results are the following:

\begin{teo}\label{thm1}
Assume that $V\geq 0, V\in L^\infty(\R^n)$ and $\displaystyle\lim_{|x|\to \infty}|x|^{2s}V(x)=0$. Then, for $$p>\frac{n+2s-1}{n-2s-1},$$
 problem \eqref{problem-1} has a continuum of solutions $u_\lambda$ such that $\displaystyle\lim_{\lambda\to 0}u_\lambda=0$ uniformly in $\R^n$.
\end{teo}

\begin{teo}\label{thm2}
Assume $V\geq 0, V\in L^\infty(\R^n)$ and
$$\frac{n+2s}{n-2s}<p\leq \frac{n+2s-1}{n-2s-1}.$$
 Then the result of the previous theorem holds if either:
\begin{itemize}
\item[(a)] there exist $C>0$ and $\mu>n$ such that $V(x)\leq C|x|^{-\mu}$; or

\item[(b)] there exist a bounded non-negative function $f:\mathbb S^{n-1}\to \R$ not identically zero and $n-\frac{4s}{p-1}<\mu\leq n$ such that
\begin{equation*}
\lim_{|x|\to \infty}\Big(|x|^\mu V-f\Big(\frac{x}{|x|}\Big)\Big)=0.
\end{equation*}
\end{itemize}
\end{teo}

Our approach follows closely the local $s=1$ case from \cite{ddmw}. As in their work, we find a new phenomenon for equation \eqref{problem-1} that is different from the subcritical case, the one of \emph{dispersion}.

The main idea in the proof is to perturb a carefully chosen approximate solution.
More precisely, the building blocks we use are smooth entire radial solutions of the equation
\begin{equation}\label{entiresolution}
(-\Delta)^s w=w^p \mbox{ in }\R^n,\quad w>0,
\end{equation}
which satisfy
\begin{equation*}
w(0)=1, \quad \lim_{|x|\to \infty}w(x)|x|^{\frac{2s}{p-1}}=\beta^{\frac{1}{p-1}}.
\end{equation*}
Here $\beta$ is  a positive constant chosen so that $w_1(r)=\beta^{\frac{1}{p-1}}|x|^{-\frac{2s}{p-1}}$ is a solution to \eqref{entiresolution}. Its precise value is given in \cite{acdfgw} (there the constant is denoted by $A_{n,p,s}$).

Note that $w_1$ is singular at the origin. One of our main contributions here is the construction of a smooth, radially symmetric, entire solution $w$. While in the local case one readily obtains  existence by simple ODE phase-plane analysis, in the nonlocal regime we need to use an argument due to Schaaf and a bifurcation argument. Then we prove the asymptotic behavior of $w$ using conformal geometry methods. These ideas were introduced in \cite{acdfgw} to handle a nonlocal ODE. There the authors deal with the singular fractional Yamabe problem and, more generally, with problem \eqref{entiresolution} when the exponent $p$ is subcritical; we were able to extend their methods to the supercritical case.

In the case particular $s=\frac{1}{2}$, Chipot, Chlebik and Fila  \cite{Chipot-Chlebik-Fila} proved the existence of a radial entire solution and Harada \cite{Harada} established the desired asymptotic behavior at infinity. While our  manuscript was being prepared we heard of the recent work by Chen, Gui and Hu \cite{cgh}, that have extended the existence result to systems which includes the existence of solutions to the single equation for any $s\in(0,1)$ using Rabinowitz's bifurcation theory. In addition, they prove an upper bound for the solution, i.e.
\begin{equation*}
w(x)\leq C|x|^{-\frac{2s}{p-1}} \mbox{ for \ all }x\in \R^n.
\end{equation*}

However, here we need to show the precise asymptotic behavior of the solution $w$ near infinity, this is,
\begin{equation}\label{asymptotics-introduction}
w=w(|x|)=\beta^{\frac{1}{p-1}}|x|^{-\frac{2s}{p-1}}(1+o(1)) \mbox{ as }|x|\to \infty,
\end{equation}
and this is our first difficulty. In the classical case, this is reduced to analyze the corresponding ODE for $v(t)$, where $v(t)=r^{\frac{2}{p-1}}w(r)$, and we have denoted $r=|x|$, $t=-\log r$. As we have mentioned above, in the fractional case $v$ is a solution to a non-local ODE. We use the Hamiltonian approach developed in the previous work \cite{acdfgw} in order to prove \eqref{asymptotics-introduction}.

Our second main difficulty is the study the linearized operator of \eqref{entiresolution} around the entire solution $w$. Here again we use the arguments developed in \cite{acdfgw}, in particular, the study of the asymptotic behavior of solutions to a non-local linear ODE in terms of the indicial roots of the problem, and the injectivity properties of this linearized operator.

Finally, to pass from the linearized equation to the nonlinear problem \eqref{problem-1}  we follows the ideas from the local case in \cite{ddmw}.

\medskip

The paper is organized as follows. In Sections 2-3, we obtain the existence of entire radial solutions to \eqref{entiresolution} and analyze its asymptotic behavior. In Section 4, we consider the linearized operator around the solution $w$ obtained in Section 2. In Section 5, we use the results in Section 4 to study a perturbed linear problem.  Sections 6-7 are devoted to the proof of Theorems \ref{thm1}-\ref{thm2}.

\section{Existence of radial entire solutions}\label{section:entire}

Let $p+1>2^*(s):=\frac{2n}{n-2s}$. We look first for radially symmetric solutions for
\begin{equation}\label{eq:LE}
\Ds{w}=w^p\quad\textin\R^n,\quad w>0.
\end{equation}
The constant $\beta>0$ is chosen as in the introduction, so that $w_1(r)=\beta^{\frac{1}{p-1}}r^{-\frac{2s}{p-1}}$ is a solution.  However, this $w_1$ is singular at the origin. Our main result in this section is the construction of an entire solution:

\begin{proposition}\label{prop:existentire}
There exists a radially symmetric, decreasing in the radial variable, entire solution $\bar w$ to \eqref{eq:LE} such that
\begin{equation*}
\bar w(r)\leq Cr^{-\frac{2s}{p-1}} \mbox{ as } r\to+\infty.
\end{equation*}
\end{proposition}

To prove this proposition, we follow the method in \cite{acdfgw}.  First we consider the Dirichlet problem
\begin{equation}\label{eq:auxprob}\begin{cases}
\Ds{w}=\lambda(1+w)^p&\textin{B_1},\\
w=0&\textin\R^n\setminus{B_1}.
\end{cases}\end{equation}
The existence for this equation follows from classical arguments (we refer to \cite{RosOton-Serra2} for details): for $\lambda>0$ small enough, since $(1-|x|^2)_+^s$ is a positive super-solution, there exists a minimal solution $w_\lambda$. Moreover, $w_\lambda$ is non-decreasing in $\lambda$. Thus one can find a $\lambda^*>0$ such that (i) the minimal solution $w_\lambda$ exists for each $\lambda\in(0,\lambda^*)$, and $w_\lambda$ is radially symmetric and decreasing in the radial variable; (ii) for $\lambda>\lambda^*$, \eqref{eq:auxprob} has no solutions.

An argument of Schaaf \cite{Schaaf} shows  uniqueness:

\begin{lemma}[Uniqueness]\label{lem:uniqueness}
There exists $\lambda_0>0$ depending only on $n$, $s$ and $p$, such that for all small $\lambda\in[0,\lambda_0)$, $w_\lambda$ is the unique solution to \eqref{eq:auxprob}.% in the space $\mathcal C^2(\R^n)\cap{L}^1(\R^n,(1+|x|)^{-n-2s}\,dx)$.
\end{lemma}

The proof this lemma is postponed to the end of the section.

\medskip

Now we consider the integral formulation of \eqref{eq:auxprob},
\begin{equation}\label{eq:auxprob2}
\begin{cases}
w=\lambda T(w)
    &\textin B_1,\\
w=0
    &\textin \R^n\setminus B_1,
\end{cases}
\end{equation}
where (see for example \cite{Bucur}),
\[T(w)(x)=C_{n,s}\int_{B_1}
    \left(\int_0^{r_0(x,y)}
        \dfrac{t^{s-1}}{(t+1)^{\frac{n}{2}}}\,dt
    \right)
    \dfrac{(1+w(y))^{p}}{|x-y|^{n-2s}}\,dy,
\quad
r_0(x,y)=\dfrac{(1-|x|^2)(1-|y|^2)}{|x-y|^2}.
\]
We perform a bifurcation argument in the Banach space of non-negative and radially symmetric, non-increasing functions supported on $B_1$,
\begin{equation*}
E=\set{w\in{\mathcal C}(\R^n):w(x)=\tilde{w}(|x|),\,\tilde{w}(r_1)\leq \tilde{w}(r_2)\text{ for } r_1>r_2,\,w\geq0\textin{B_1},\,w=0\textin\R^n\setminus{B_1}}.
\end{equation*}

\begin{lemma}[Bifurcation]\label{lem:bifurcation}
There exists a sequence of solutions $(\lambda_j,w_j)$ of \eqref{eq:auxprob} in $(0,\lambda^*]\times{E}$ such that
\[\lim_{j\to\infty}\lambda_j=\lambda_\infty\in[\lambda_0,\lambda^*]
\quad\textand\quad
\lim_{j\to\infty}\norm[L^\infty(B_1)]{w_j}=+\infty,\]
where $\lambda_0$ is given in Lemma \ref{lem:uniqueness}.
\end{lemma}

\begin{proof}
The proof
%of Lemma \ref{lem:bifurcation}
follows closely \cite[Lemma 2.8]{acdfgw}. In the integral formulation \eqref{eq:auxprob2},
it is readily checked that
%$T(w):=(-\Delta)^{-s}((1+w)^p)$
$T$ defines a (nonlinear) compact operator. Moreover, the pair $(0,0)$ is a solution. Now, \cite[Theorem 6.2]{Rabinowitz} yields an unbounded continuum of solutions in $\R \times E$ connecting $(0,0)$. Since the minimal solution $w_\lambda$ is unique for $\lambda\in(0,\lambda_0)$ by Lemma \ref{lem:uniqueness} and there does not exist any solution for $\lambda>\lambda^*$, we must have $\lambda\in[\lambda_0,\lambda^*]$ whenever $\norm[L^\infty(B_1)]{w_\lambda}$ is large. %By Schauder estimates, the continuum is also unbounded in the supremum norm.
\end{proof}

Now we are ready to prove Proposition \ref{prop:existentire} using a blow-up argument.

%The proofs of lemmata \ref{lem:uniqueness} and \ref{lem:bifurcation} will be given later in this section.

\begin{proof}[Proof of Proposition \ref{prop:existentire}]
Let $(\lambda_j,w_j)$ be as in Lemma \ref{lem:bifurcation}. Define
\[m_j=\norm[L^{\infty}(B_1)]{w_j}=w_j(0)
\quad\textand\quad
R_j=m_j^{\frac{p-1}{2s}}.\]
Then
\[W_j(x)=\lambda_j^{\frac{1}{p-1}}m_j^{-1}w_j\left(\frac{x}{R_j}\right)\]
satisfies $0\leq{W_j}\leq{W_j(0)}=1$ and
\[\begin{cases}
\Ds{W_j}=\left(\lambda_j^{\frac{1}{p-1}}m_j^{-1}+W_j\right)^p
    &\textin{B_{R_j}},\\
W_j=0
    &\textin\R^n\setminus{B_{R_j}}.
\end{cases}\]
By elliptic regularity (see for instance \cite{RosOton-Serra2}), $W_j\in\mathcal {C}_{\loc}^{\alpha}(\R^n)$ for some $\alpha>0$ and hence, by passing to a subsequence, $W_j\to{\bar w}$ in $\mathcal C_\loc(\R^n)$ for some positive and radially decreasing function $\bar w$ which satisfies \eqref{eq:LE}. Moreover, the solution $\bar w$ is radially symmetric and decreasing. In addition, from Lemma 2.6 of \cite{acdfgw} with $\beta=0$ there, one can get that
\begin{equation*}
\bar w(x)\leq C|x|^{-\frac{2s}{p-1}} \mbox{ for }x\in \mathbb R^n.
\end{equation*}
\end{proof}

We finally go back to uniqueness and the proof of Lemma \ref{lem:uniqueness}:

\begin{proof}[Proof of Lemma \ref{lem:uniqueness}]
Suppose $w=w_\lambda+v$ is another solution. Then $v$ solves
\[\begin{cases}
\Ds{v}=\lambda{f(v)}&\textin{B_1},\\
v=0&\textin\R^n\setminus{B_1},
\end{cases}\]
where $f(v)=(1+w_\lambda+v)^p-(1+w_\lambda)^p$. Let us write $F(v)=\int_0^{v}f$. For any $\sigma\in\R$, the Poho\v{z}aev identity \cite{RosOton-Serra1} reads
\begin{equation*}%\label{eq:Poho}
\left(\dfrac{1}{2^*(s)}-\sigma\right)\int_{\R^n}v\Ds{v}\,dx\leq\lambda\int_{B_1}\left(F(v)-\sigma{v}f(v)\right)\,dx.
\end{equation*}
The left hand side is estimated using the fractional Sobolev inequality so that
\begin{equation*}%\label{eq:Poho1}
\left(\dfrac{1}{2^*(s)}-\sigma\right)\left(\int_{B_1}v^{2^*(s)}\,dx\right)^{\frac{2}{2^*(s)}}\,dx\leq\lambda{C}\int_{B_1}\left(F(v)-\sigma{v}f(v)\right)\,dx.
\end{equation*}
For the right hand side, since $\displaystyle\lim_{v\to+\infty}\frac{F(v)}{vf(v)}=\frac{1}{p+1}$,
for any $\eps>0$ there exists $M=M(\eps)$ such that for any $v\geq{M}$, $F(v)\leq\frac{1+\eps}{p+1}vf(v)$. Since $p+1>2^*(s)$, there exist $\eps$ and $\sigma$ such that $\frac{1+\eps}{p+1}<\sigma<\frac{1}{2^*(s)}$. In other words, fixing a choice of $\eps$ and $\sigma$, there is an $M$ such that $F(v)-\sigma{v}f(v)<0$ whenever $v\geq{M}$. Now, since $F(v)$ is quadratic in $v$, namely,
\[F(v)=v^2\int_0^1\int_0^1pt_1(1+w_\lambda+t_1t_2v)^{p-1}\,dt_1dt_2,\]
we have
\begin{equation*}%\label{eq:Poho1}
\begin{split}
\left(\dfrac{1}{2^*(s)}-\sigma\right)\left(\int_{B_1}v^{2^*(s)}\,dx\right)^{\frac{2}{2^*(s)}}\,dx
&\leq\lambda{C}\int_{B_1\cap\set{v<M}}\left(F(v)-\sigma{v}f(v)\right)\,dx\\
&\leq\lambda{C}\int_{B_1\cap\set{v<M}}v^2\,dx\\
&\leq\lambda{C}\left(\int_{B_1\cap\set{v<M}}v^{2^*(s)}\,dx\right)^{\frac{2}{2^*(s)}}|B_1\cap\set{v<M}|^{1-\frac{2}{2^*(s)}}\\
&\leq\lambda{C}\left(\int_{B_1\cap\set{v<M}}v^{2^*(s)}\,dx\right)^{\frac{2}{2^*(s)}}.
\end{split}
\end{equation*}
Therefore, $v\equiv0$ if $\lambda>0$ is chosen small enough.
\end{proof}

\section{Conformal geometry results and the asymptotic behavior for the solution}

From the last section, we know that there exists a smooth solution $\bar w$ of $(-\Delta)^s w=w^p $ in $\R^n $ which satisfies $\bar{w}(0)=1$ and $\bar{w}(x)\leq C|x|^{-\frac{2s}{p-1}}$. Here we will show that it has a precise asymptotic behavior at infinity. For the rest of the paper, we will drop the bar and simply denote this particular solution by $w$. More precisely,

\begin{proposition}\label{prop:asymptotics}
Let $w$ be the solution constructed in Proposition \ref{prop:existentire}. Then
\begin{equation*}
\lim_{|x|\to \infty}w(x)|x|^{\frac{2s}{p-1}}=\beta^{\frac{1}{p-1}},
\end{equation*}
for the constant $\beta$ given in the introduction, and which corresponds to the coefficient of the singular solution to \eqref{eq:LE}.
\end{proposition}

We will basically use the results  in Sections 3-4 of \cite{acdfgw}, and follow their notation accordingly. Set $r=|x|$, and let $P_s^{g_0}$ be the conformal fractional Laplacian on the cylinder $\mathbb R\times\mathbb S^{n-1}$. Then, by its conformal properties one has that
\begin{equation}\label{conjugation1}
P_s^{g_0}(r^{\frac{n-2s}{2}}w)=r^{\frac{n+2s}{2}}(-\Delta)^{s} w.
\end{equation}
Set also $r=e^{-t}$ and
\begin{equation}\label{wv}
v=e^{-\frac{2s}{p-1}t}w(e^{-t}),
\end{equation}
and consider one further conjugation
\begin{equation*}\label{tilde-P}
\tilde P_s^{g_0}(v):=
e^{(\frac{n-2s}{2}-\frac{2s}{p-1})t}P_s^{g_0}\big(e^{(-\frac{n-2s}{2}+\frac{2s}{p-1})t}v\big)
=r^{\frac{2s}{p-1}p}(-\Delta)^s w.
\end{equation*}
Then the equation for $w$ transforms into an equation for $v$ on the cylinder, more precisely,
\begin{equation}\label{conformal-laplacian}
\tilde P_s^{g_0}(v)=v^p, \quad\text{in } \mathbb R \times \mathbb S^{n-1}.
\end{equation}

The operator $\tilde P_s^{g_0}$ can be understood as a Dirichlet-to-Neumann for an extension problem in the spirit of the construction of the fractional Laplacian by \cite{Caffarelli-Silvestre,Chang-Gonzalez,Gonzalez:survey,DelaTorre-Gonzalez}. Without being very precise on the extension manifold $X^{n+1}$, with metric $\bar g^*$, and the extension variable $\rho^*$, we recall the following proposition in \cite{acdfgw}:

\begin{proposition} \label{prop:divV*}
Let $v$ be a smooth function on the cylinder $M=\mathbb R\times\mathbb S^{n-1}$.
The extension problem
\begin{equation*}\label{divV*}
\left\{\begin{array}{@{}r@{}l@{}l}
 -\divergence_{\bar{g}^*}((\rho^*)^{1-2s}\nabla_{\bar{g}^*}V^*)
-(\rho^*)^{-(1+2s)}\left(\tfrac{4\rho}{4+\rho^2}\right)^{2}2\Big(-\frac{n-2s}{2}+\tfrac{2s}{p-1}\Big)\,\partial_t V^*&\,=0\quad &\text{in } (X,\bar g^*), \medskip\\
V^*|_{\rho=0}&\,=v\quad &\text{on }M,
\end{array}\right.
\end{equation*}
has a unique solution $V^*$. Here $\rho=\rho(\rho^*)$ is a positive function of $\rho^*$, $\rho^*\in(0,\rho_0^*)$.
 Moreover, for its Neumann data,
\begin{equation*}\label{Neumann-V*}
\tilde P_s^{g_0}(v)=-d_s \lim_{\rho^* \to 0}(\rho^*)^{1-2s} \partial_{\rho^*} (V^*)+\beta v,
\end{equation*}
for the constant
\begin{equation}\label{constant-extension}
d_s=\frac{2^{2s-1}\Gamma(s)}{s\Gamma(-s)},
\end{equation}
and $\beta$ as above.
\end{proposition}

Now consider the spherical harmonic decomposition of $\mathbb S^{n-1}$. For this, let $\mu_m$  denote the $m$-th eigenvalue for $-\Delta_{\mathbb S^{n-1}}$, repeated according to multiplicity, and by $E_m(\theta)$ the corresponding eigenfunction. Then any $v$ defined on $\mathbb R\times\mathbb S^{n-1}$ can be written as $v(t,\theta)=\sum_m v_m(t)E_m(\theta)$. Denote by $\tilde P_s^{(m)}$ the projection of the operator $\tilde P_s^{g_0}$ onto each eigenspace, for $m=0,1,\ldots$.

Next, for $w$ a radially symmetric function, then $v$ defined as in  \eqref{wv} only depends on the variable $t=-\log r$ and thus, after projection onto spherical harmonics, only the zero-th projection survives and equation \eqref{conformal-laplacian} reduces to
\begin{equation}\label{problem-vv}
\tilde P_s^{(0)}(v)=v^p, \quad\text{in } \mathbb R.
\end{equation}
Recalling the expression for $\tilde P^{(0)}_\gamma$ from \cite[formula (4.3)]{acdfgw}, one has
\begin{equation}\label{convolution-K}
\tilde P^{(0)}_s(v)(t)=P.V. \int_{\mathbb R}\tilde{\mathcal K}_0(t-t')[v(t)-v(t')]\,dt'+\beta v(t)
\end{equation}
for the convolution kernel
\begin{equation*}
\tilde{\mathcal K}_0(t)=c\,e^{-(\frac{2s}{p-1}-\frac{n-2s}{2})t}\int_0^\pi \frac{(\sin \phi_1)^{n-2}}{(\cosh t-\cos \phi_1)^{\frac{n+2s}{2}}}\,d\phi_1,
\end{equation*}
where $c$ is a positive constant that only depends on $n$ and $s$.

As a consequence of Proposition \ref{prop:divV*}, one obtains the following Hamiltonian quantity for the equation \eqref{problem-vv} (Theorem 4.3 in \cite{acdfgw}),
\begin{equation*}\label{Hamiltonian}
\begin{split}
H^*(t)=&\frac{1}{{d}_{s}}
\left(-\frac{\beta}{2}v^2+\frac{1}{p+1}v^{p+1}\right)
+\frac{1}{2}\int_{0}^{\rho^*_0}  {(\rho^*)}^{1-2\gamma}\left\{-e^*_1(\partial_{\rho^*}V^*)^2
+e^*_2(\partial_t V^*)^2\right\}\,d\rho^*\\
=&:H_1(t)+H_2(t),
\end{split}
\end{equation*}
where $V^*=V^*(t,\rho^*)$ is the extension of $v=v(t)$. Moreover, this Hamiltonian is monotone in $t$. Indeed,
\begin{equation}\label{ham_der0}
\partial_t\left[ H^*(t)\right]=-2\int_{0}^{\rho^*_0}\left[(\rho^*)^{1-2s}e^*
\left(\tfrac{\rho}{\rho^*}\right)^{-2}\left(1+\tfrac{\rho^2}{4}\right)^{-2}
\left(-\tfrac{n-2s}{2}+\tfrac{2s}{p-1}\right)\left[\partial_t V^*\right]^2\right]\,d\rho^*,
\end{equation}
and the functions $e^*$, $e_1^*$ and $e^*_2$ are strictly positive.

\bigskip

\begin{proof}[Proof of Proposition \ref{prop:asymptotics}] Our aim is to show that
\begin{equation*}
\lim_{t\to -\infty}v(t)=\beta^{\frac{1}{p-1}}.
\end{equation*}
Equation \eqref{ham_der0} implies that
\begin{equation*}\label{ham_der}
\partial_t\left[ H^*(t)\right]
\geq 0,
\end{equation*}
since $\left(-\tfrac{n-2s}{2}+\tfrac{2s}{p-1}\right)<0$ due to the fact that $p>\frac{n+2s}{n-2s}$. Thus in this case the  Hamiltonian $H^*(t)$ is increasing in $t$. By Proposition  \ref{prop:existentire},  $v$ is uniformly bounded, hence so is $V^*$. By elliptic regularity,  $\partial_t V^*, \ \partial_{\rho^*}V^*$ are also bounded. Then $H^*$ is uniformly bounded. Integrating \eqref{ham_der0} on the real line $\R$, one has
\begin{equation*}
\begin{split}
\int_{-\infty}^\infty \partial_t\left[ H^*(t)\right]dt&=-2\int_{-\infty}^\infty\int_{0}^{\rho^*_0}\left[(\rho^*)^{1-2\gamma}e^*
\left(\tfrac{\rho}{\rho^*}\right)^{-2}\left(1+\tfrac{\rho^2}{4}\right)^{-2}
\left(-\tfrac{n-2s}{2}+\tfrac{2s}{p-1}\right)\left[\partial_t V^*\right]^2\right]\,d\rho^*dt\\
&=H^*(+\infty)-H^*(-\infty),
\end{split}
\end{equation*}
which is finite by the uniform boundedness of $H^*$. Thus we have that $\int_0^{\rho_0^*}g(\rho^*)|\partial_t V^*|^2\,d\rho^*\to 0 $ as $t\to \pm\infty$ for some non-negative function $g$, which implies that $\partial_t V^*\to 0$ as $t\to \pm \infty$.

For any sequence $\{t_i\}\to -\infty$, consider $V^*_i(t, \rho^*)=V^*(t+t_i, \rho^*)$. There exists a subsequence $V^*_i(t, \rho^*)\to V_\infty(t, \rho^*)$ in $\mathcal{C}\left([-1, 1]\times [0, \rho_0^*]\right)$. By the argument above,
we know that $\partial_t V_\infty(t, \rho^*)=0$, i.e. $V_\infty=V_\infty(\rho^*)$, and this is a solution of the same equation satisfied by $V^*(t, \rho^*)$. Since $V^*$ is the extension of $v(t)$,  one has $v(t)\to V_\infty(0)$ as $t\to -\infty$, where $V_\infty(0)$ is a constant.  Moreover, $V_\infty(0)$ is a solution of
\begin{equation*}
\tilde P_s^{(0)}(v)=v^p, \quad\text{in } \mathbb R.
\end{equation*}
Since $\displaystyle\lim_{t\to +\infty}H^*(t)=0$, $\partial_t H^*(t)\geq 0$ and is not identically zero. One then obtains that $V_\infty(\rho^*)$ is not trivial, so $V_\infty(0)$ is not zero.  From this and the expression of $\tilde P^{(0)}_s$ above (formula \eqref{convolution-K}), one can easily get that $V_\infty(0)=\beta^{\frac{1}{p-1}}$, as desired.
\end{proof}

%%%%%%%%%%%%%%%%%%%%%%%%%%%%%%%%%%%%%%%%%%%%%%%%%%%%%%%%%%%%%%%%%%%%%%%%%%%%%%%%%%%%%%%%%%%%%%%%%%%%%%%%%%%%%%%%%%%%%%%

\section{Linear theory}\label{section:linear}

From the last section, we know that there exists a solution $w$ of $(-\Delta)^s u=u^p $ in $\R^n $ that satisfies
\begin{equation}\label{w}
\left\{\begin{array}{l}
w(0)=1,\\
w(x)\sim\beta^{\frac{1}{p-1}}|x|^{-\frac{2s}{p-1}},\quad \text{as } |x|\to \infty.
\end{array}
\right.
\end{equation}
We are now interested in the linear problem:
\begin{equation}\label{model-linearization}
L\phi:=(-\Delta)^s\phi-\frac{\mathcal V(r)}{r^{2s}}\phi=h,
\end{equation}
where we have defined the (radial) potential
\begin{equation*}\label{potential}
\mathcal V:=\mathcal V(r)=r^{2s}pw^{p-1}.
\end{equation*}
 The asymptotic behavior of this potential is easily calculated using the asymptotics for $w$
 and, indeed, if we define $r=e^{-t}$,
\begin{equation}\label{asymptotics-potential}
\mathcal V(t)=\begin{cases}
O(e^{-2st}) &\mbox{ as }t\to +\infty,\\
p\beta +O(e^{q_1t} ) &\mbox{ as }t\to -\infty,
\end{cases}\end{equation}
for some $q_1>0$.

In the $s=1$ case, solutions to \eqref{model-linearization} are constructed directly by projecting $\phi$ and $h$ onto spherical harmonics and then solving the corresponding ODEs. This approach cannot be directly applied here. Instead, we use the conformal geometry tools developed in \cite{acdfgw}.

By the well known extension theorem for the fractional Laplacian  \cite{Caffarelli-Silvestre},  equation
\eqref{model-linearization} is also equivalent, for $s\in(0,1)$, to the boundary reaction problem
\begin{equation*}\label{eqlinear1}
\left\{\begin{array}{r@{}l@{}l}
\partial_{yy}\Phi+\dfrac{1-2s}{y}\partial_y \Phi+\Delta_{\R^n}\Phi
    &\,=0
    &\quad\mbox{in }\R^{n+1}_+,\medskip\\
- d_s\lim\limits_{y\to0}y^{1-2s}\partial_y \Phi
    &\,=\dfrac{\mathcal V(r)}{r^{2s}}\phi+h
    &\quad\mbox{on }\R^n,
\end{array}\right.
\end{equation*}
where we have denoted $\Phi|_{y=0}=\phi$ and the constant $d_s$ is given by \eqref{constant-extension}.

Using the spherical harmonic decomposition as above, we can write $\Phi=\sum_{m=0}^\infty \Phi_m(r,y)E_m(\theta)$, $\phi=\sum_{m=0}^\infty \phi_m(r)E_m(\theta)$, $h=\sum_{m=0}^\infty h_m(r)E_m(\theta)$, and $\Phi_m$ satisfies the following:
\begin{equation}\label{eqlinear2}
\left\{\begin{array}{r@{}l@{}l}
\partial_{yy}\Phi_{m}+\dfrac{1-2s}{y}\partial_y\Phi_{m}
+\Delta_{\R^n}\Phi_m-\dfrac{\mu_m}{r^2}\Phi_m
    &\,=0
    &\quad\mbox{in }\R^{n+1}_+,\medskip\\
- d_s\lim\limits_{y\to0}y^{1-2s}\partial_y \Phi_m
    &\,=\dfrac{\mathcal V(r)}{r^{2s}}\phi_m+h_m
    &\quad\mbox{on }\R^n.
\end{array}
\right.
\end{equation}

We will now use conformal geometry to rewrite equation \eqref{model-linearization}, as explained in the previous section. If we  define
\begin{equation}\label{shift}
\psi=r^{\frac{n-2s}{2}}\phi,
\end{equation}
then by the conformal property \eqref{conjugation1}, we have that
 equation \eqref{model-linearization} is equivalent to the problem
\begin{equation}\label{eq0}
\mathcal L \psi:=P_s^{g_0}(\psi)-\mathcal V\psi=r^{2s}h=:\tilde h.
\end{equation}
Projecting onto spherical harmonics we also have
\begin{equation}\label{eq0m}
\mathcal L_m \psi_m:=P_s^{(m)}(\psi_m)-\mathcal V\psi_m=\tilde h_m, \quad m=0,1,\ldots.
\end{equation}

\subsection{Indicial roots}\label{subsection:indicial-roots}

Let us calculate the indicial roots for the model linearized operator defined in \eqref{model-linearization} as $r\to 0$ and as $r\to \infty$.

One can see that the indicial roots here are similar to the ones in \cite{acdfgw} with the roles at $r=0$ and $\infty$ interchanged. However, recall that in \cite{acdfgw} a subcritical power is taken, $\frac{n}{n-2s}<p<\frac{n+2s}{n-2s}$, while here $p>\frac{n+2s}{n-2s}$ is supercritical. To handle this difference we need to use the result in \cite{lwz}, where the authors study the stability of the singular solution $w=\beta^{\frac{1}{p-1}}r^{-\frac{2s}{p-1}}$. In particular, they show the existence of a threshold dimension $n_0(s)$ such that in any higher dimension $n>n_0(s)$, there exists $p_2>\frac{n+2s}{n-2s}$ such that the singular solution is unstable if $\frac{n+2s}{n-2s}<p<p_2$, and is stable if $p\geq p_2$. For $n\leq n_0(s)$, however, the singular solution is unstable for all $p>\frac{n+2s}{n-2s}$. When $s=1$, this exponent $p_2$ corresponds to the well-known Joseph-Lundgren exponent \cite{gnw}.

%it is shown that there exists $n_0(s)>0$, satisfying that for all $n>n_0(s)$, then there exists $p_2>\frac{n+2s}{n-2s}$ such that the singular solution is unstable if $\frac{n+2s}{n-2s}<p<p_2$, and is stable if $p\geq p_2$, while for $n\leq n_0(s)$, for all $p>\frac{n+2s}{n-2s}$, the solution is unstable.  When $s=1$, this exponent $p_2$ corresponds to the well-known Joseph-Lundgren exponent \cite{gnw}.

Define now
\begin{equation*}
P_{JL}=\left\{\begin{array}{l}
p_2, \quad n>n_0(s),\\
\infty, \quad n\leq n_0(s).
\end{array}
\right.
\end{equation*}

Using the above result and following similar argument in the proof of Lemma 7.1 in \cite{acdfgw}, one has the following result:

\begin{lemma}\label{indicial}
For the operator $L$ we have that, for each fixed mode $m=0,1,\ldots,$
\begin{itemize}
\item[\emph{i.}] At $r=0$, there exist two sequences of indicial roots
\begin{equation*}
\Big\{\tilde \sigma_j^{(m)}\pm i\tilde\tau_j^{(m)}-\tfrac{n-2s}{2}\Big\}_{j=0}^\infty\quad \text{and} \quad\Big\{-\tilde\sigma_j^{(m)}\pm i\tilde\tau_j^{(m)}-\tfrac{n-2s}{2}\Big\}_{j=0}^\infty.
\end{equation*}
Moreover, at the $j=0$ level the indicial roots are real; more precisely, the numbers
\begin{equation*}
\tilde\gamma_m^{\pm}:=\pm\tilde\sigma_0^{(m)}-\tfrac{n-2s}{2}=-\tfrac{n-2s}{2}\pm \left[1-s+\sqrt{(\tfrac{n-2}{2})^2+\mu_m}\right],\quad m=0,1,\ldots,
\end{equation*}
%and
%it
%$\tilde{\gamma}_m^+$
%is
form an increasing sequence in $m$ (except for multiplicity repetitions).

\item[\emph{ii.}] At $r=\infty$,  there exist two sequences of indicial roots
\begin{equation*}
\Big\{\sigma_j^{(m)}\pm i\tau_j^{(m)}-\tfrac{n-2s}{2}\Big\}_{j=0}^\infty\quad \text{and} \quad\Big\{-\sigma_j^{(m)}\pm i\tau_j^{(m)}-\tfrac{n-2s}{2}\Big\}_{j=0}^\infty.
\end{equation*}
Moreover,
\begin{itemize}
\item[\emph{a)}] For the mode $m=0$, if $\frac{n+2s}{n-2s}<p<P_{JL}$ (the \underline{unstable} case), then the indicial roots $\gamma_0^\pm$ are a pair of complex conjugates with real part $-\frac{n-2s}{2}$ and imaginary part $\pm\tau_0^{(0)}$;
while if $p\geq P_{JL}$ (the \underline{stable} case), the indicial roots $\gamma_0^{\pm}:=\pm\sigma_0^{(0)}-\frac{n-2s}{2}$ are real with
\begin{equation*}\label{gamma0}
-(n-2s)+\tfrac{2s}{p-1}<\gamma_0^-<-\tfrac{n-2s}{2}<\gamma_0^+<-\tfrac{2s}{p-1},
\end{equation*}

\item[\emph{b)}] In addition, for all $j\geq 1$,
\begin{equation*}
\sigma_j^{(0)}>\tfrac{n-2s}{2}.
\end{equation*}

\item[\emph{c)}] %For the mode $m=1$, we have the following relations:
%\begin{equation*}\label{indicial1}
%\gamma_1^{\pm}:=\pm \sigma_0^{(1)}-\tfrac{n-2s}{2}
%=-\tfrac{2s}{p-1}-1 \quad\mbox{ and }\quad -(n-2s)+\tfrac{2s}{p-1}+1.
%\end{equation*}
%{\color{red} Do you mean:\\
For the mode $m=1$, the indicial roots
$\gamma_1^{\pm}:=\pm \sigma_0^{(1)}-\tfrac{n-2s}{2}$ take the values
$$-\tfrac{2s}{p-1}-1 \quad\mbox{ and }\quad -(n-2s)+\tfrac{2s}{p-1}+1.$$
%?}yes
%Good! This is now the new wording that I think is clearer.
\end{itemize}
\end{itemize}
\end{lemma}

\subsection{Conformal geometry and invertibility for a Hardy-type operator}

First we recall the results in \cite{acdfgw}, adapted to our setting, in order to study the invertibility of a Hardy type operator with fractional Laplacian:
\begin{equation*}
\mathcal L_\kappa \psi:=P_s^{g_0}(\psi)-\kappa\psi=r^{2s}h=:\tilde h,
\end{equation*}
where $\kappa$ is any fixed real constant. Projecting onto spherical harmonics this equation is equivalent to
\begin{equation}\label{eq-kappa}
\mathcal L_{\kappa,m} \psi_m:=P_s^{(m)}(\psi_m)-\kappa\psi_m=\tilde h_m, \quad m=0,1,\ldots.
\end{equation}

We will consider values of $\kappa$ for which the indicial roots of $\mathcal L_{\kappa,m}$ are those of Lemma \ref{indicial}, except for a shift of $\frac{n-2s}{2}$ due to the change \eqref{shift}. Based on the results in \cite{acdfgw} we have, in both the stable and the unstable cases of Lemma \ref{indicial}, %that:
the following

\begin{teo}\label{thm:Hardy-potential}
Fix $m=0,1,\ldots$, and assume  that the right hand side $\tilde h_m$ in \eqref{eq-kappa} satisfies
\begin{equation*}\label{decay-h}
h_m(t)=\begin{cases}
O(e^{\delta t}) &\text{as } t\to-\infty,\\
O(e^{-\delta_0 t}) &\text{as } t\to+\infty,
\end{cases}
\end{equation*}
for some real constants $\delta,\delta_0$.
\begin{itemize}
 \item[\emph{i.}] Assume that $\delta+\delta_0\geq 0$. Then a particular solution to \eqref{eq-kappa} is
    \begin{equation*}\label{solucion2}
\psi_m(t)=\int_{\mathbb R}\tilde h_m(t') {{\mathcal G}}_m(t-t')\,dt',
\end{equation*}
where $\mathcal G$ is an even $\mathcal C^\infty$ function outside the origin. The exact formula for $\mathcal G_m$  depends on the value of $\delta$ with respect to the indicial roots in Lemma \ref{indicial}.  In any case,
\begin{equation}\label{decay-phi}
\psi_m(t)=O(e^{\delta t}) \quad\text{as} \quad t\to -\infty, \quad \psi_m(t)=O(e^{-\delta_0 t}) \quad\text{as} \quad t\to +\infty.
\end{equation}

\item[\emph{ii.}] All solutions of the homogeneous problem $\mathcal L_{\kappa,m} \psi_m=0$ are linear combinations of $e^{\sigma_j\pm i\tau_j}$ and $e^{-\sigma_j\pm i\tau_j}$, where $\sigma_j\pm i\tau_j$, $-\sigma_j\pm i\tau_j$, $j=0,1,\ldots$, are the indicial roots for the problem. Thus the only solution to \eqref{eq-kappa} with decay as in \eqref{decay-phi} is precisely $\psi_m$.
\end{itemize}
\end{teo}

\subsection{Study of the linear problem in weighted spaces}

In this subsection we come to study the linear problem \eqref{model-linearization}.  For this, we will work in weighted $L^\infty$ spaces in which the weight is chosen differently in a bounded set and near infinity. Define the norms
\begin{equation}\label{norms}
\begin{split}
\|\phi\|_*&=\sup_{\{|x|\leq 1\}}|x|^\sigma |\phi(x)|+\sup_{\{|x|\geq 1\}}|x|^{\frac{2s}{p-1}}|\phi(x)|,\\
\|h\|_{**}&=\sup_{\{|x|\leq 1\}}|x|^{\sigma+2s} |h(x)|+\sup_{\{|x|\geq 1\}}|x|^{\frac{2s}{p-1}+2s}|h(x)|,
\end{split}
\end{equation}
where $\sigma\in (0, n-2s)$ is a constant to be determined later.

Our aim  is to get the following solvability result:

\begin{proposition}\label{invertibility}
Let $h$ satisfy $\|h\|_{**}<+\infty$. For linear problem \eqref{model-linearization},  we have:
\begin{itemize}
\item[\emph{i}.]
if $p>\frac{n+2s-1}{n-2s-1}$, then there exists a solution $\phi$ and it satisfies
\begin{equation}\label{bound}
\|\phi\|_*\leq C\|h\|_{**}
\end{equation}
for some $C>0$;

\item[\emph{ii.}]
if $\frac{n+2s}{n-2s}<p<\frac{n+2s-1}{n-2s-1}$ and, in addition,
\begin{equation*}
\int_{\R^n}h\frac{\partial w}{\partial x_i}dx=0, \ i=1,\cdots, n,
\end{equation*}
then there exists a solution $\phi$ and it satisfies
\begin{equation*}
\|\phi\|_*\leq C\|h\|_{**}.
\end{equation*}
\end{itemize}
\end{proposition}

\begin{remark}
It is known that the linear operator $L$ has $n+1$ kernels corresponding to scaling (the mode zero kernel) and translation (the mode one kernels), i.e.
\begin{equation*}
\begin{split}
z_0(x)&=\frac{\partial }{\partial \lambda}\Big( \lambda^{\frac{2s}{p-1}}w(\lambda x)\Big)\Big|_{\lambda=1}=rw'(r)+\frac{2s}{p-1}w, \\
z_i(x)&=\frac{\partial }{\partial x_i}w(x),\quad i=1,\ldots,n.
\end{split}
\end{equation*}
These constitute an obstruction for the solvability of \eqref{model-linearization} and need to be taken into account in the arguments below.
\end{remark}

In the following, we will always assume  that
\begin{equation*}
\sigma \in \big(0, |\Real(\gamma_0^+)|\big)\subset \Big(0, \frac{n-2s}{2}\Big).
\end{equation*}
Additional conditions will be given in the proofs below.

\medskip

We start with a non-degeneracy result:

\begin{lemma}\label{non-degeneracy}
If $\phi$ is a solution of
\begin{equation*}
(-\Delta)^s \phi-pw^{p-1}\phi=0 \quad\mbox{in }\R^n
\end{equation*}
satisfying  $\|\phi\|_*<\infty$, then
\begin{equation*}
\phi=\sum_{i=1}^n c_i\frac{\partial w}{\partial x_i}
\end{equation*}
for some $c_i\in \R$.
\end{lemma}
\begin{proof}
 Consider the spherical harmonic decomposition  $\displaystyle\phi=\sum_m \phi_m E_m$ and write $\psi_m=r^{\frac{n-2s}{2}}\phi_m$.

{\bf Step 1:} the mode $m=0$. Define the constant $\kappa=p\beta$ and rewrite equation \eqref{eq0m} for $m=0$ as
\begin{equation}\label{choice-h}
\mathcal L_{\kappa,0}\psi:=P_s^{(0)}(\psi)-\kappa \psi=(\mathcal V-\kappa)\psi=:\tilde h
\end{equation}
for some $\psi=\psi(t)$, $\tilde h=\tilde h(t)$. We use \eqref{asymptotics-potential} and the definition of $\psi$ to estimate the right hand side,
\begin{equation*}\label{asymptotics-h}
\tilde h(t)=\begin{cases}
 O(e^{-(\frac{n-2s}{2}-\sigma)t} )&\mbox{ as }t\to +\infty,\\
 O(e^{(q_1+\frac{2s}{p-1}-\frac{n-2s}{2})t})&\mbox{ as }t\to -\infty,
\end{cases}\end{equation*}
for some $q_1>0$.
Note that there could be solutions to the homogeneous problem of the form $e^{(\sigma_j\pm i\tau_j)t}$, $e^{(-\sigma_j\pm i\tau_j)t}$. But these are not allowed by the choice of weights since $\frac{n-2s}{2}-\frac{2s}{p-1}\in (0, \frac{n-2s}{2})$ and $\frac{n-2s}{2}-\sigma>\sigma_0^{(0)}$ (for this, recall statements \emph{a)} and \emph{b)} in Lemma \ref{indicial}).

Now we apply Theorem \ref{thm:Hardy-potential} with $\delta=q_1+\frac{2s}{p-1}-\frac{n-2s}{2}<-\sigma_0^{(0)}$ and $\delta_0=-\sigma+\frac{n-2s}{2}>\sigma_0^{(0)}$. Obviously, $\delta+\delta_0>0$ if {$\sigma<\frac{2s}{p-1}$.} Then we can find a particular solution $\psi_0$ such that
\begin{equation*}
\psi_0(t)=(e^{\delta t}), \quad\text{as}\quad t\to -\infty,\qquad
\psi_0(t)=(e^{-\delta_0 t}), \quad\text{as}\quad t\to +\infty,
\end{equation*}
so $\psi$ will have the same decay.

Now, by the definition of $h$ in \eqref{choice-h}, we can iterate this process with $\delta=lq_1+\frac{2s}{p-1}-\frac{n-2s}{2}$, $l\geq 2$, and the same $\delta_0$, to obtain better  decay when $t\to -\infty$. As a consequence, we have that $\psi$ decays faster than any $e^{\delta t}$ as $t\to -\infty$, which when translated to $\phi$ means that $\phi=o(r^{-a})$ as $r\to +\infty$ for every $a\in\mathbb N$. By considering the equation satisfied by the Kelvin transform $\hat{\phi}$ of $\phi$, one has
\begin{equation*}
(-\Delta)^s \hat{\phi}-\frac{\tilde{\mathcal{V}}}{r^{2s}}\hat{\phi}=0,
\end{equation*}
where $\tilde{\mathcal{V}} $ satisfies
\begin{equation*}
\tilde{\mathcal{V}}(x)=p|x|^{-2s}w^{p-1}\Big(\frac{x}{|x|^2}\Big)=\left\{\begin{array}{ll}
p\beta(1+o(1)) &\mbox{ as }r\to 0,\\
r^{-2s} &\mbox{ as } r\to \infty,
\end{array}
\right.
\end{equation*}
and $\hat{\phi}(r)=o(r^{a})$ as $r\to 0$ for every $a\in\mathbb N$. The strong unique continuation result of \cite{Fall-Felli} (stable case) and \cite{Ruland} (unstable case) for the operator $P_s^{(0)}-\hat{\mathcal V}$ implies that $\hat{\phi}$ must vanish everywhere, which yields that also $\phi$ must be zero everywhere.

\medskip

{\bf Step 2:} the modes $m=1,\ldots,n$.   Differentiating equation $(-\Delta)^s w=w^p$
with respect to $x_m$ we get
\begin{equation*}
 L\frac{\partial w}{\partial x_m}=0.
\end{equation*}
Since $w$ only depends on $r$, we have $\frac{\partial w}{\partial x_m}=w'(r)E_m$, where $E_m=\frac{x_m}{|x|}$. Using the fact that $-\Delta_{\mathbb S^{n-1}} E_m=\mu_m E_m$, the extension for $w'(r)$ to $\mathbb R^{n+1}_+$
%, call it $\overline U$, will solve
solves \eqref{eqlinear2} with eigenvalue $\mu_m=n-1$, and $\psi_1:=r^{\frac{n-2s}{2}}w'$ satisfies \begin{equation}\label{eq}
P_s^{(m)} \psi-\mathcal V\psi=0.
\end{equation} Note that $w'(r)$ decays like $r^{-(\frac{2s}{p-1}+1)}$ as $r\to \infty$ and decays like $r$ as $r\to 0$.

 Assume that $\phi_m$ decays like $r$ as $r\to 0$ and decays like $r^{\gamma_m^-}$ as $r\to \infty$. Then  $\psi_m=r^{\frac{n-2s}{2}}\phi_m$ is another solution to \eqref{eq}, and  we can find a non-trivial combination of $w'$ and $\phi_m$ that decays faster than $r$ at zero.

Now we claim that if $\phi=r^{-\frac{n-2s}{2}}\psi$, where $\psi$ is any solution to \eqref{eq}, it cannot decay faster than
%$r^{-(\frac{2s}{p-1}+1)}$ at $\infty$,  and faster than
$r$ at $0$, which yields that $\phi_m=cw'$ for $m=1, \ldots, n$.

To show this claim we argue as in Step 1, taking the indicial roots at $0$ (namely $-(n+1-2s)$ and $1$) and interchanging the role of $+\infty$ and $-\infty$ in the decay estimate. More precisely, we use the facts that if such $\phi$ like $r^{\sigma'}$ for some $\sigma'>1$, i.e. $\sigma'+\frac{n-2s}{2}>\frac{n-2s}{2}+1=\tilde{\sigma}_0^{(1)}$ and decays like $r^{\sigma'_1}$ for some $\sigma'_1<-(\frac{2s}{p-1}+1)$, i.e. $\sigma'_1+\frac{n-2s}{2}<\tilde{\sigma}_1^{(1)}$, similarly to Step 1, one can show that the solution is identically zero, and we conclude that $\phi_m=cw'$ for some $c$.

\medskip

{\bf Step 3:} the remaining modes $m\geq{n+1}$. We use an integral estimate involving the first mode which has a sign, as in \cite{Davila-delPino-Musso-Wei,Davila-delPino-Musso}. %\[\phi(x)=\sum_{m=0}^{\infty}\phi_{m}(r)E_m(\theta).\]
We note, in particular, that $\phi_1(r)=-w'(r)>0$, which also implies that its extension $\Phi_1$ is positive. In general, the $s$-harmonic extension $\Phi_m$ of $\phi_m$ satisfies
\[\left\{
\begin{array}{r@{}ll}
\divergence(y^{1-2s}\nabla\Phi_m)
    &\,=\mu_{m}\dfrac{y^{1-2s}}{r^2}\Phi_m
    &\text{ in } \R^{n+1}_+,\medskip\\
-{d}_s\lim\limits_{y\to0}y^{1-2s}\p_{y}\Phi_m
    &\,=pu_1^{p-1}\phi_m
    &\text { on } \R^{n+1}_+.
\end{array}
\right.\]
We multiply this equation by $\Phi_1$ and the one for $m=1$ by $\Phi_m$. Their difference gives the equality
\begin{equation*}\begin{split}
(\mu_m-\mu_1)\dfrac{y^{1-2s}}{r^2}\Phi_m\Phi_1
&=\Phi_1\divergence(y^{1-2s}\nabla\Phi_m)-\Phi_m\divergence(y^{1-2s}\nabla\Phi_1)\\
&=\divergence(y^{1-2s}(\Phi_1\nabla\Phi_m-\Phi_m\nabla\Phi_1)).
\end{split}\end{equation*}
Let us integrate over the region where $\Phi_m>0$. The boundary $\p\set{\Phi_m>0}$ is decomposed into a disjoint union of $\p^0\set{\Phi_m>0}$ and $\p^+\set{\Phi_m>0}$, on which $y=0$ and $y>0$, respectively. Hence
\begin{equation*}\begin{split}
0&\leq{d}_s(\mu_m-\mu_1)\int_{\set{\Phi_m>0}}\dfrac{y^{1-2s}\Phi_m\Phi_1}{r^2}\,dxdy\\
&=\int_{\p^0\set{\Phi_m>0}}\left(\phi_1\lim\limits_{y\to0}y^{1-2s}\pnu{\Phi_m}-\phi_m\lim\limits_{y\to0}y^{1-2s}\pnu{\Phi_1}\right)\,dx\\
&+\int_{\p^+\set{\Phi_m>0}}y^{1-2s}\left(\Phi_1\pnu{\Phi_m}-\Phi_m\pnu{\Phi_1}\right)
\,dy.
\end{split}\end{equation*}
The first integral on the right hand side vanishes due to the equations $\Phi_1$ and $\Phi_m$ satisfy. Then we observe that on $\p^+\set{\Phi_m>0}$, one has $\Phi_1>0$, $\pnu{\Phi_m}\leq0$ and $\Phi_m=0$. This forces (using $\mu_m>\mu_1$)
\[\int_{\set{\Phi_m>0}}\dfrac{y^{1-2s}\Phi_m\Phi_1}{r^2}\,dxdy=0,\] which in turn implies $\Phi_m\leq0$. Similarly $\Phi_m\geq0$ and, therefore, $\Phi_m\equiv0$ for $m\geq{n+1}$. This completes the proof of the lemma.
\end{proof}

Now we turn to Fredholm properties. Let $L$ be the operator defined in \eqref{model-linearization}. It is actually simpler to consider the conjugate operator $\mathcal L$ defined in \eqref{eq0} (and its projection \eqref{eq0m}), which is better behaved in  weighted Hilbert spaces and simplifies the notation in the proof of Fredholm properties.

We define weighted $L^2_{\delta,\vartheta}$ function spaces. These contain $L^2_{\text{loc}}$ functions for which the norm
\begin{equation}\label{L2-norm}
\|\phi\|^2_{L^2_{\delta,\vartheta}(\mathbb R^n)}=\int_{\mathbb R^n\setminus B_1} |\phi|^2r^{ n-1-2s-2\vartheta}\,drd\theta +\int_{B_1}|\phi|^2r^{n-1-2s-2\delta}\,drd\theta
\end{equation}
is finite. These should be understood after conjugation \eqref{shift}, as
\begin{equation*}
\|\psi\|^2_{L^2_{\delta,\vartheta}(\mathbb R\times\mathbb S^{n-1})}=\int_{-\infty}^0\int_{\mathbb S^{n-1}} |\psi|^2e^{2\vartheta t}\,d\theta  dt+\int_{0}^{\infty}\int_{\mathbb S^{n-1}}|\psi|^2e^{2\delta t}\,d\theta dt.
\end{equation*}
The spaces $L^2_{\delta,\vartheta}$ and $L^2_{-\delta,-\vartheta}$ are dual with respect to the natural pairing
\begin{equation*}
\langle\psi_1,\psi_2\rangle_*=\int_{\mathbb R^n} \psi_1 \psi_2,
\end{equation*}
for  $\psi_1\in L^2_{\delta,\vartheta}$, $\psi_2\in L^2_{-\delta,-\vartheta}$.

The relation between \eqref{L2-norm} and the weighted $L^\infty$ norms from \eqref{norms} is given by the following simple lemma:

\begin{lemma}\label{lemma:inclusions}
Assume that the parameters satisfy
\begin{equation}\label{parameters}
-\delta<-\sigma+\frac{n-2s}{2},\quad -\vartheta>-\frac{2s}{p-1}+\frac{n-2s}{2},
\end{equation}
Then if $\|\phi\|_*<\infty$, one has that $\|\phi\|_{L^2_{-\delta,-\vartheta}}$ is also finite and the inclusion is continuous.
\end{lemma}

\begin{proof}[Proof of Proposition \ref{invertibility}]   First note that elliptic estimates imply that $\mathcal L$ is a densely defined, closed graph operator. Moreover, the adjoint of
\begin{equation*}\label{operator1}
\mathcal L:L^2_{-\delta,-\vartheta}\to L^2_{-\delta,-\vartheta}
\end{equation*}
is precisely
\begin{equation*}\label{operator2}
\mathcal L^*=\mathcal L:L^2_{\delta,\vartheta}\to L^2_{\delta,\vartheta}.
\end{equation*}

Similarly to the arguments in Section 8 of \cite{acdfgw}, one can show that the linear operator $\mathcal L$ satisfies good Fredholm properties and, in particular,
$$\Ker(\mathcal L^*)^\bot=\rango( \mathcal L).$$

By checking the proof for Lemma \ref{non-degeneracy}, one can also get that for this linear problem $\mathcal L^*\phi=0$, the mode $0$ and mode $m$ for $m\geq n+1$ all have trivial kernels. For mode $1$, there is a one dimensional solution spanned by $w'(r)$.  It is easy to see that $w'(r)\in L^2_{\delta,\vartheta}$ iff $p<\frac{n+2s-1}{n-2s-1}$, if one chooses the parameters as in \eqref{parameters}, since
\begin{equation*}
\begin{split}
\|\phi\|^2_{L^2_{\delta,\vartheta}(\mathbb R^n)}&=\int_{\mathbb R^n\setminus B_1} |\phi|^2r^{ n-1-2s-2\vartheta}\,drd\theta +\int_{B_1}|\phi|^2r^{n-1-2s-2\delta}\,drd\theta\\
&\sim \int_1^\infty r^{-\frac{4s}{p-1}-2}r^{n-1-2s-\frac{4s}{p-1}+n-2s}\,dr+\int_0^1 r^2r^{n-1-2s-2\sigma+n-2s}\,dr<+\infty
\end{split}
\end{equation*}
for suitable small $\sigma>0$ if  $p<\frac{n+2s-1}{n-2s-1}$.

From the above argument and the Fredholm property we  conclude that:
\begin{itemize}
\item If $\frac{n+2s}{n-2s}<p<\frac{n+2s-1}{n-2s-1}$, then
\begin{equation*}
L\phi=h
\end{equation*}
is solvable iff $\int_{\mathbb R^n}h\frac{\partial w}{\partial x_i}\,dx=0$ for $i=1, \cdots,n$.

\item If $p>\frac{n+2s-1}{n-2s-1}$, then
\begin{equation*}
L\phi=h
\end{equation*}
is always solvable.
\end{itemize}
Moreover, the Fredholm property yields that
\begin{equation}\label{L^2}
\|\phi\|_{L^2_{\delta,\vartheta}(\mathbb R^n)}\leq C\|h\|_{L^2_{\delta-2s,\vartheta-2s}(\mathbb R^n)}.
\end{equation}
We will show that this estimate still holds in weighted $L^\infty$ norm, i.e. \eqref{bound} holds.

As in \cite{acdfgw}, combining with the Green's representation formula, one has
\begin{equation*}
\phi(x)=-\int_{\R^n}G(x,y)pw^{p-1}\phi(y)\,dy-\int_{\R^n}G(x,y)h(y)\,dy.
\end{equation*}
First, for  $|x|\leq 1$,
\begin{equation*}
\begin{split}
\int_{\R^n}G(x,y)h(y)\,dy&\leq C\left[\int_{\{|y|<\frac{|x|}{2}\}}+\int_{\{\frac{|x|}{2}\leq |y|\leq 2|x|\}}+\int_{\{|y|>2|x|\}}\right]\frac{h(y)}{|x-y|^{n-2s}}\,dy\\
&\leq C\left[ \int_{\{|y|<\frac{|x|}{2}\}}\frac{\|h\|_{**}|y|^{-\sigma-2s}}{|x|^{n-2s}}\,dy+\int_{\{|y|\leq 3|x|\}}\frac{\|h\|_{**}|x|^{-\sigma-2s}}{|y|^{n-2s}}\,dy\right.\\
&\quad+\left.\int_{\{2>|y|>2|x|\}}
\frac{\|h\|_{**}|y|^{-\sigma-2s}}{|y|^{n-2s}} \,dy+\int_{\{|y|\geq 2\}}\frac{\|h\|_{**}|y|^{-\frac{2s}{p-1}-2s}}{|y|^{n-2s}}\,dy\right]\\
&\leq C|x|^{-\sigma}\|h\|_{**}.
\end{split}
\end{equation*}
Next, by the definition of weighted $L^2$ norm and relations \eqref{parameters} and \eqref{L^2},
\begin{equation*}
\begin{split}
\int_{\R^n}&G(x,y)pw^{p-1}\phi(y)\,dy\\&\leq C\left[\int_{\{|y|<\frac{|x|}{2}\}}+\int_{\{\frac{|x|}{2}\leq |y|\leq 2|x|\}}+\int_{\{2>|y|>2|x|\}}\frac{\phi}{|x-y|^{n-2s}}\,dy+\int_{\{|y|>2\}}\frac{|y|^{-2s}\phi}{|x-y|^{n-2s}}\,dy  \right]\\
&\leq C\Big( \int_{\{|y|<2\}}\phi^2 r^{-2s-2\delta}\,dy\Big)^{\frac{1}{2}}\\
&\quad\cdot \left[  \Big(\int_{\{|y|<\frac{|x|}{2}\}}\frac{|y|^{2s+2\delta}}{|x|^{2(n-2s)}}\,dy\Big)^{\frac{1}{2}}
+\Big(\int_{\{|y|<3|x|\}}\frac{|x|^{2s+2\delta}}{|y|^{2(n-2s)}}\,dy\Big)^{\frac{1}{2}} +(\int_{\{2|x|<|y|<2\}} \frac{|y|^{2s+2\delta}}{|y|^{2(n-2s)}}\, dy\Big)^{\frac{1}{2}}\right]\\
&+C\Big(\int_{\{|y|>2\}}\phi^2r^{-2s-2\vartheta}dy\Big)^{\frac{1}{2}}\Big(\int_{\{|y|>2\}}
\frac{|y|^{-4s+2s+2\vartheta}}{|y|^{2(n-2s)}}\,dy \Big)^{\frac{1}{2}}\\
&\leq C\|\phi\|_{L_{\delta,\vartheta}^2(\R^n)}|x|^{-\sigma}\leq C\|h\|_{L_{\delta-2s,\vartheta-2s}^2(\R^n)}|x|^{-\sigma}\\
&\leq C\|h\|_{**}|x|^{-\sigma},
\end{split}
\end{equation*}
for suitable $\delta, \vartheta$ satisfying \eqref{parameters}. Thus one has
\begin{equation*}
\sup_{\{|x|\leq 1\}}|x|^{\sigma}|\phi|\leq C\|h\|_{**}.
\end{equation*}
Similarly, for $|x|\leq e^{R}$, we still obtain
\begin{equation*}
\sup_{\{1\leq |x|\leq e^R\}}|x|^{\frac{2s}{p-1}}|\phi|\leq C_R\|h\|_{**},
\end{equation*}
for some constant depending on $R$. This implies that  the weighted $L^\infty$ norm of $\phi$ in any compact set can be bounded by the weighted $L^\infty$ norm of $h$. So one only needs to worry about the norm at infinity.  For this, we go back to the projected problems:
\begin{equation*}
\mathcal L_{\kappa,m} \psi_m:=P_s^{(m)}(\psi_m)-\kappa\psi_m=\tilde{h}_m+(\mathcal V-\kappa)\psi_m, \quad m=0,1,\ldots.
\end{equation*}
In order to show the estimate for $\phi$ at infinity, it is enough to prove that
\begin{equation*}
\|e^{\frac{-2s}{p-1}t}\psi_m\|_{L^\infty(\{t\leq -R\})}\leq C\|e^{\frac{-2s}{p-1}t}\tilde{h}_m\|_{L^\infty(\{t\leq 0\})}+\|e^{-\sigma t}\tilde{h}_m\|_{L^\infty(\{t\geq 0\})}
\end{equation*}
for $R$ large enough. But this follows from  Theorem \ref{thm:Hardy-potential} and the expression for $\mathcal V$ in (\ref{asymptotics-potential}) by taking $\delta=\frac{2s}{p-1}$ and $\delta_0=-\sigma$ (See Lemma 6.7 and the proof for Proposition 6.3 in \cite{acdfgw}). Here we use the assumption on $\sigma$ that $0<\sigma<\frac{2s}{p-1}$ such that $\delta+\delta_0>0$.

This completes the proof of the Proposition.
\end{proof}

\section{The operator $(-\Delta)^s \phi+V_\lambda \phi-pw^{p-1}\phi$ in $\R^n$}

In this section we study the following linear problem in $\R^n$ in suitable weighted spaces:
\begin{equation}\label{lp}
\begin{cases}
(-\Delta)^s \phi+V_\lambda \phi-pw^{p-1}\phi
=h+\sum\limits_{i=1}^n c_iZ_i,\\
\lim\limits_{|x|\to \infty}\phi(x)=0,
\end{cases}
\end{equation}
where $c_i$ are real constants, and $Z_i$ are the kernels defined by
\begin{equation}\label{kernel}
Z_i(x)=\frac{\partial w}{\partial x_i}, \ i=1,\cdots,n,
\end{equation}
 and
\begin{equation*}
V_\lambda(x)=\lambda^{-2s}V\Big(\frac{x-\xi}{\lambda}\Big) \mbox{ for }\lambda>0, \ \xi \in \R^n.
\end{equation*}
Since $V_\lambda$ has a concentration  point at $\xi$, we define
\begin{equation*}
\begin{split}
\|\phi\|_{*,\xi}&=\sup_{\{|x-\xi|\leq 1\}}|x-\xi|^\sigma |\phi(x)|+\sup_{\{|x-\xi|>1\}}|x-\xi|^{\frac{2s}{p-1}}|\phi(x)|,\\
\|h\|_{**,\xi}&=\sup_{\{|x-\xi|\leq 1\}}|x-\xi|^{\sigma+2s} |h(x)|+\sup_{\{|x-\xi|>1\}}|x-\xi|^{\frac{2s}{p-1}+2s}|h(x)|.
\end{split}
\end{equation*}
We will take $|\xi|\leq \Lambda$ for some $\Lambda>0$.

For the linear theory, it suffices to assume that
\begin{equation}\label{assumptionv}
V\in L^\infty(\R^n), \quad V\geq 0, \quad V(x)=o(|x|^{-2s}) \mbox{ as }|x|\to \infty.
\end{equation}
Then we have the following solvability result:
\begin{proposition}\label{pro1}
We have:

\begin{itemize}

\item
If $p>\frac{n+2s-1}{n-2s-1}$, for $\lambda>0$ small enough, equation \eqref{lp} has a solution $\phi:=\mathcal{T}_\lambda(h)$ with $c_i=0$, and it satisfies
\begin{equation*}
\|\mathcal{T}_\lambda(h)\|_{*,\xi}\leq C\|h\|_{**,\xi}.
\end{equation*}
\item
If $\frac{n+2s}{n-2s}<p<\frac{n+2s-1}{n-2s-1}$, for $\lambda>0$ small enough, equation \eqref{lp} has a solution $(\phi,c_1,\cdots,c_n):=\mathcal{T}_\lambda(h)$, and it satisfies
\begin{equation*}
\|\phi\|_{*,\xi}+\sum_{i=1}^n|c_i|\leq C\|h\|_{**,\xi}.
\end{equation*}
The constant $C>0$ is independent of the parameter $\lambda$.
\end{itemize}
\end{proposition}

\begin{proof}
The proof follows the argument in Section 3 of \cite{ddmw}. So here we only sketch the proof. We solve the linear problem near the point $\xi$ and away from this point. For this, decompose $\phi=\varphi+\psi$, where $\varphi, \psi$ satisfy
\begin{equation}\label{varphi}
\begin{cases}
\displaystyle(-\Delta)^s\varphi-pw^{p-1}\varphi=p\xi_0w^{p-1}\psi-\xi_1V_\lambda\varphi+\xi_1h+\sum_{i=1}^nc_iZ_i,\\
\lim\limits_{|x|\to \infty}\varphi(x)=0,
\end{cases}
\end{equation}
and
\begin{equation}\label{psi}
\begin{cases}
(-\Delta)^s\psi-p(1-\xi_0)w^{p-1}\psi+V_\lambda\psi=-(1-\xi_1)V_\lambda\varphi+(1-\xi_1)h,\\
\lim\limits_{|x|\to \infty}\varphi(x)=0,
\end{cases}
\end{equation}
where $\xi_0, \xi_1$ are two cut off functions:
\begin{equation*}
\xi_0(x)=0 \mbox{ for }|x-\xi|\leq R, \quad \xi_0(x)=1 \mbox{ for }|x-\xi|\geq 2R,
\end{equation*}
and
\begin{equation*}
\xi_1(x)=0 \mbox{ for }|x-\xi|\leq \varrho, \quad \xi_1(x)=1 \mbox{ for }|x-\xi|\geq 2\varrho,
\end{equation*}
for $\varrho,R$  two positive constants independent of $\lambda$ to be fixed later and such that $2\varrho\leq R$.

Given $\|\varphi\|_{*,\xi}<\infty$, since $\|p(1-\xi_0)w^{p-1}\|_{L^{\frac{n}{2s}}}\to 0$ as $R\to 0$, equation \eqref{psi} has a solution $\psi=\psi(\varphi)$ if $R>0$ is small enough (this is because the homogeneous problem for \eqref{psi} has only the trivial solution).
Moreover, $\psi(x)=O(|x|^{-(n-2s)})$ as $|x|\to \infty$, so the right hand side of \eqref{varphi} has finite $\|\cdot\|_{**}$ norm, by Proposition \ref{invertibility}, \eqref{varphi} has a solution when $\psi=\psi(\varphi)$ which we write as $F(\varphi)$. We claim that $F(\varphi)$ has a fixed point in the Banach space
\begin{equation*}
X=\{\varphi\in L^\infty(\R^n), \|\varphi\|_*<\infty\}
\end{equation*}
equipped with
\begin{equation*}
\|\varphi\|_X=\sup_{\{|x|\leq 1\}}|\varphi|+\sup_{\{|x|\geq 1\}}|x|^{\frac{2s}{p-1}}|\varphi|.
\end{equation*}

Following the argument in the proof of Proposition 3.1 in \cite{ddmw}, we establish pointwise estimates for the solution $\psi$ of \eqref{psi}. Then we can find a bound of the $\|\cdot\|_{**}$ norm of the right hand side of \eqref{varphi}. Since the proof is similar, we omit the details and just state the estimates here:
\begin{equation*}
\begin{split}
\|\psi\|_{*,\xi}&\leq (C_\varrho+C\varrho^{n-2s})\|\varphi\|_X+C_\varrho\|h\|_{**,\xi},\\
\|\xi_0w^{p-1}\psi(\varphi)\|_{**,\xi}&\leq C\varrho^{n-2s}\|\varphi\|_X+C_\varrho\|h\|_{**,\xi},\\
\|\xi_1 V_\lambda \varphi\|_{**,\xi}&\leq C\varrho^{-2s}\|\varphi\|_X a\big(\frac{\varrho}{\lambda}\big),
\end{split}
\end{equation*}
where
\begin{equation*}
a(r)=\sup_{\{|x|\geq r\}}|x|^{2s}V(x), \quad a(r) \to 0\mbox{ as }r\to \infty.
\end{equation*}

By the linear theory in Section 2, we know that given $\varphi\in X$, the solution $F(\varphi)$ to (\ref{varphi}) satisfies
\begin{equation*}
\|F(\varphi)\|_{*,\xi}\leq C\|\xi_0w^{p-1}\psi(\varphi)\|_{**,\xi}+C\|\xi_1 V_\lambda \varphi\|_{**,\xi}+C\|\xi_1h\|_{**,\xi}.
\end{equation*}
But since the right hand side of (\ref{varphi}) is bounded near the origin, by regularity estimates, we derive

\begin{equation*}
\|F(\varphi)\|_X\leq C\|\xi_0w^{p-1}\psi(\varphi)\|_{**,\xi}+C\|\xi_1 V_\lambda \varphi\|_{**,\xi}+C\|\xi_1h\|_{**,\xi}
\end{equation*}
and
\begin{equation*}
\begin{split}
\|F(\varphi_1)-F(\varphi_2)\|_X&\leq C\varrho^{n-2s}\|\varphi_1-\varphi_2\|_X+C a(\frac{\varrho}{\lambda})\varrho^{-2s}\|\varphi_1-\varphi_2\|_X\\
&\leq C\big(\varrho^{n-2s}+\varrho^{-2s}a\big(\frac{\varrho}{\lambda}\big)\big)\|\varphi_1-\varphi_2\|_X.
\end{split}
\end{equation*}
By choosing $\varrho>0 $ small enough and $\lambda$ small enough such that $\frac{\varrho}{\lambda}\to \infty$, we can prove that $F(\varphi)$ is a contraction mapping, and we get a fixed point $\varphi\in X$. Moreover, thanks to the linear theory in Section \ref{section:linear}, we have
\begin{equation*}
\|\varphi\|_X\leq C\big(\varrho^{n-2s}+\varrho^{-2s}a\big(\frac{\varrho}{\lambda}\big)\big)\|\varphi\|_X+C_\varrho \|h\|_{**,\xi},
\end{equation*}
which yields
\begin{equation*}
\|\varphi\|_X\leq C\|h\|_{**,\xi}.
\end{equation*}
Combining with the estimate for $\psi$, one has
\begin{equation*}
\|\phi\|_{*,\xi}\leq C\|h\|_{**,\xi}
\end{equation*}
for some $C>0$ independent of $\lambda $ small.

\end{proof}

\section{Proof of Theorem \ref{thm1}}

Assume that $p>\frac{n+2s-1}{n-2s-1}$. We aim to find a solution to \eqref{problem-1} of the form $u=w+\phi$, where $w$ is the radial entire solution found in Section \ref{section:entire} (recall that $w$ satisfies \eqref{w}). This  yields the following equation for the function $\phi$,
\begin{equation*}
(-\Delta)^s\phi+V_\lambda \phi-pw^{p-1}\phi=N(\phi)-V_\lambda w,
\end{equation*}
where
\begin{equation*}
N(\phi)=(w+\phi)^p-w^p-pw^{p-1}\phi.
\end{equation*}

Using the operator $\mathcal{T}_\lambda$ defined in Proposition \ref{pro1}, we are led to solve the fixed point problem:
\begin{equation*}\label{fixedpoint}
\phi=\mathcal{T}_\lambda(N(\phi)-V_\lambda w):=\mathcal{A}(\phi).
\end{equation*}
Following the argument in Section 4 of \cite{ddmw}, one obtains the following estimate for $N(\phi)$ and $V_\lambda w$:
\begin{equation*}
\|V_\lambda w\|_{**}\leq C\Big(\lambda^\sigma R^{2s+\sigma}\|V\|_{L^\infty}+a(R)+a\big(\frac{1}{\lambda}\big)\Big)\leq Ca(R)
\end{equation*}
as $\lambda\to 0$. Choosing $R\to \infty$, then as $\lambda\to 0$ one has that $\|V_\lambda w\|_{**}\to 0$. In addition, for the nonlinear term  $N(\phi)$ we have the bound
\begin{equation*}
\|N(\phi)\|_{**}\leq C(\|\phi\|_*^2+\|\phi\|_*^p).
\end{equation*}

Consider the set
\begin{equation*}
\mathcal{F}=\{\phi:\R^n\to \R, \|\phi\|_*\leq \rho\}
\end{equation*}
where $\rho>0$ small is to be chosen later.

It is standard to get the following estimates:
\begin{equation*}
\begin{split}
\|\mathcal{A}(\phi)\|_*&\leq C(\|\phi\|_*^2+\|\phi\|_*^p+\|V_\lambda w\|_{**})<\rho,\\
\|\mathcal{A}(\phi_1)-\mathcal{A}(\phi_2)\|_*&\leq C(\|\phi_1\|^{\min\{p-1,1\}}_*+\|\phi_2\|^{\min\{p-1,1\}}_*)\|\phi_1-\phi_2\|_*<\frac{1}{2}\|\phi_1-\phi_2\|_*
\end{split}
\end{equation*}
for $\lambda $ small and $\rho$ small. One can find that for $\rho$ small enough, $\mathcal{A}(\phi)$ is a contraction mapping in $\mathcal{F}$, thus has a fixed point in this set. This finishes the proof of the theorem.
\qed

\section{The case $\frac{n+2s}{n-2s}<p\leq \frac{n+2s-1}{n-2s-1}$}

In this case, because of the presence of the kernels $Z_i$ defined in \eqref{kernel}, one needs to introduce free parameters and rescale around a point $\xi$, to be chosen later. For this reason, we make the change of variable $\lambda^{-\frac{2s}{p-1}}u(\frac{x-\xi}{\lambda})$ and look for a solution of the form $u=w+\phi$, then $\phi$ satisfies
\begin{equation*}
(-\Delta)^s \phi+V_\lambda \phi-pw^{p-1}\phi=N(\phi)-V_\lambda w,
\end{equation*}
where
\begin{equation*}
V_\lambda(x)=\lambda^{-2s}V\big(\frac{x-\xi}{\lambda}\big).
\end{equation*}

Similarly to the proof of Lemma 5.1 in \cite{ddmw} can show the following result, for which we omit the proof:

\begin{lemma}
Let $\frac{n+2s}{n-2s}<p<\frac{n+2s-1}{n-2s-1}$ and  $\Lambda>0$. Assume that $V$ satisfies \eqref{assumptionv}. Then there exists $\ve_0>0$ such that for $|\xi|<\Lambda$ and $\lambda<\ve_0$, there exists $(\phi, c_1, \cdots, c_n) $ solution of
\begin{equation}\label{nonlinear1}
\left\{\begin{array}{l}
\displaystyle(-\Delta)^s \phi+V_\lambda \phi-pw^{p-1}\phi=N(\phi)-V_\lambda w+\sum_{i=1}^n c_iZ_i,\\
\displaystyle\lim_{|x|\to \infty}\phi(x)=0.
\end{array}
\right.
\end{equation}
We have, in addition,
\begin{equation*}
\|\phi\|_{*,\xi}+\sum_{i=1}^n |c_i|\to 0 \mbox{ as }\lambda\to 0.
\end{equation*}

If $V$ also satisfies
\begin{equation*}\label{assumptionv1}
V(x)\leq C|x|^{-\mu} \mbox{ for \ all }x
\end{equation*}
for some $\mu>2s$, then for $0<\sigma\leq \mu-2s, \sigma<n-2s$, one has
\begin{equation*}
\|\phi\|_{*,\xi}\leq C_\sigma \lambda^\sigma \mbox{ for \ all }0<\lambda<\ve_0.
\end{equation*}
\end{lemma}

\bigskip

Now we are ready for the proof of our second main theorem:

\begin{proof}[Proof of Theorem \ref{thm2}.] We have found a solution $(\phi,c_1,\cdots, c_n)$ to equation \eqref{nonlinear1}. By the linear theory in Section 2, this solution satisfies
\begin{equation*}
\int_{\R^n}\Big(N(\phi)-V_\lambda\phi-V_\lambda w+\sum_{i=1}^n c_iZ_i\Big)\frac{\partial w}{\partial x_j}\,dx=0
\end{equation*}
for $j=1,\cdots, n$. So for $\lambda$ small, we need to choose the parameter $\xi$ such that $c_i=0$ for all $i$, that is
\begin{equation}\label{reduceproblem}
\int_{\R^n}\Big(N(\phi)-V_\lambda\phi-V_\lambda w\Big)\frac{\partial w}{\partial x_j}\,dx=0
\end{equation}
because the matrix with coefficients $\int_{\R^n }Z_i\frac{\partial w}{\partial x_j}\,dx$ is invertible.

{\bf Case a.} Since $V\leq |x|^{-\mu}$ for $\mu>n$, the dominant term in \eqref{reduceproblem} is
\begin{equation*}
\lambda^{-2s}\int_{\R^n}V\big(\frac{x-\xi}{\lambda}\big)w\frac{\partial w}{\partial x_j}\,dx=O(\lambda^{n-2s}).
\end{equation*}
Using the estimates for $\phi$ in the last section, one can get that
\begin{equation*}
\int_{\R^n}N(\phi)\frac{\partial w}{\partial x_j}\,dx=o(\lambda^{n-2s}).
\end{equation*}
For the term $\int V_\lambda \phi \frac{\partial w}{\partial x_j}\,dx$, one has
\begin{equation*}
\begin{split}
\int V_\lambda \phi \frac{\partial w}{\partial x_j}\,dx&=\int_{B_{R\lambda}}\cdots+\int_{B_1 \setminus B_{R\lambda}}\cdots+ \int_{\R^n\setminus B_1}\cdots \\
&\leq C\int_{B_{R\lambda}}\lambda R\cdot \lambda^{-2s}\cdot \lambda^\sigma |x|^{-\sigma }\,dx+\int_{B_1\setminus B_{\lambda R}}\lambda^\sigma |x^{-\sigma}|\cdot a(R)|x|^{-2s}\lambda |x|\,dx\\
&\quad +\int_{\R^n\setminus B_1}\lambda^\sigma a\big(\frac{1}{\lambda}\big)|x|^{-(2s+\frac{4s}{p-1}+1)}\,dx\\
&\leq o(\lambda^{n-2s})
\end{split}
\end{equation*}
as $\lambda \to 0$, where we have used the fact that $\frac{n+2s}{n-2s}<p<\frac{n+2s-1}{n-2s-1}$.
%Next, we consider $u_\lambda=w+\phi$ as one part, and claim that
%\begin{equation}
%c<u_\lambda<C \mbox{ in }B_1(\xi).
%\end{equation}
%From the equation satisfied by $u_\lambda$ and the estimate for $\phi_\lambda$, one can easily get that $\|u_\lambda^p\|_{L^q}$ is bounded for some $q>\frac{N}{2s}$. So by regularity, one can get the uniform upper bound for $u_\lambda$. For the lower bound, consider $\chi(x)=(1-|x|^2)^s$, it satisfies
%\begin{equation*}
%\left\{\begin{array}{l}
%(-\Delta)^s \chi=c_1 \mbox{ in }B_1\\
%\chi(x)=0 \mbox{ on }\partial B_1
%\end{array}
%\right.
%\end{equation*}
%for some $c_1>0$ constant. Thus one can check that the function $z(x)=u_\lambda+\frac{1}{c_1}\sum_{i=1}^n|c_i|\|Z_i\|_{L^\infty}\chi(x)$ satisfies $(-\Delta)^s z(x)+V_\lambda(x)z(x)\geq 0$ in $B_1$, and using the maximum principle and  the fact that $\phi_\lambda\to 0$ and $u_\lambda\to w$ in compact set in $B_1\setminus \{0\}$, $c_i\to 0$ as $\lambda\to 0$, one has $u_\lambda\geq c$ for some $c>0$ in $B_1$.

We now set
\begin{equation*}
F_\lambda^{(j)}(\xi)=\lambda^{-2s}\int_{\R^n}V\big(\frac{x}{\lambda}\big)u_\lambda\frac{\partial w}{\partial x_j}(x+\xi)\,dx+\int_{\R^n}N(\phi)\frac{\partial w}{\partial x_j}(x+\xi)\,dx.
\end{equation*}
Fix $\rho>0$ small. For $|\xi|=\rho$, and $\lambda$ small, one can check that
\begin{equation*}
\langle F_\lambda(\xi), \xi\rangle \sim \langle \nabla^2w(0) \xi,\xi\rangle<0 \mbox{ for }|\xi|=\rho,
\end{equation*}
since $0$ is local maximum point of $w$. By degree theory, we deduce that $F_\lambda$ has a zero in $B_\rho $.

\medskip

{\bf Case {b}(i). } Assume that $\displaystyle\lim_{|x|\to \infty}\Big(|x|^\mu V-f\big(\frac{x}{|x|}\big)\Big)=0$ for $N-\frac{4s}{p-1}<\mu<n, \ f\neq 0.$
Here one can check that the dominant term of \eqref{reduceproblem}  is
\begin{equation*}
\begin{split}
&\lambda^{-2s}\int_{\R^n}V(\frac{x-\xi}{\lambda})w\frac{\partial w}{\partial x_j}\,dx\\
&=\lambda^{\mu-2s}\int_{R^n}|x|^{-\mu}f\Big(\frac{x}{|x|}\Big)w(x+\xi)\frac{\partial w}{\partial x_j}(x+\xi)\,dx+o(\lambda^{\mu-2s})\\
&=O(\lambda^{\mu-2s}).
\end{split}
\end{equation*}
Again, using the estimates for $\phi$, one has
\begin{equation*}
\begin{split}
\int_{\R^n}N(\phi)\frac{\partial w}{\partial x_j}\,dx&=O(\lambda^{\mu+\sigma-2s})=o(\lambda^{\mu-2s}),\\
\int_{\R^n}V_\lambda \phi \frac{\partial w}{\partial x_j}\,dx&\leq C\lambda^{\mu+\sigma-2s}.
\end{split}
\end{equation*}
So now define
\begin{equation*}
\begin{split}
\tilde{F}_j(\xi)&:=\frac{1}{2}\int_{R^n}|x|^{-\mu}f\Big(\frac{x}{|x|}\Big)w^2(x+\xi)\,dx\\
&=\frac{1}{2}\beta^{\frac{2s}{p-1}}|\xi|^{-\frac{4s}{p-1}+n-\mu}\int_{\R^n}|y|^{-\mu}
f\Big(\frac{y}{|y|}\Big)\Big|y+\frac{\xi}{|\xi|}\Big|^{-\frac{4s}{p-1}}\,dy+o(|\xi|^{n-\mu-\frac{4s}{p-1}}).
\end{split}
\end{equation*}
Similarly,
\begin{equation*}
\begin{split}
\nabla \tilde{F}(\xi)\cdot \xi&=\frac{1}{2}\Big(n-\mu-\frac{4s}{p-1}\Big)\beta^{\frac{2s}{p-1}}|\xi|^{-\frac{4s}{p-1}+n-\mu}
\int_{\R^n}|y|^{-\mu}
f\Big(\frac{y}{|y|}\Big)\Big|y+\frac{\xi}{|\xi|}\Big|^{-\frac{4s}{p-1}}\,dy+o(|\xi|^{n-\mu-\frac{4s}{p-1}}).
\end{split}
\end{equation*}
Therefore
\begin{equation*}
\nabla \tilde{F}(\xi)\cdot \xi<0 \mbox{ \ for all } |x|=R \mbox{ \ for large }R.
\end{equation*}
Using degree theory, we get the existence of $\xi $ in $B_R$ such that $c_i=0$ for all $i$.

\medskip

{\bf Case b(ii).} $\displaystyle\lim_{|x|\to \infty}\Big(|x|^n V-f\big(\frac{x}{|x|}\big)\Big)=0,  f\neq 0.$
In this case, we have
\begin{equation*}
\begin{split}
G_j(\xi)&:=\int_{\R^n}\Big(N(\phi)-V_\lambda\phi-V_\lambda w\Big)\frac{\partial w}{\partial x_j}\,dx\\
&=\lambda^{-2s}\int_{\R^n}V(\frac{x}{\lambda})u_\lambda(x+\xi)\frac{\partial w}{\partial x_j}(x+\xi)\,dx+o(\lambda^{n-2s}).
\end{split}
\end{equation*}
We claim that $\langle G(\xi),\xi\rangle <0$ for all $|\xi|=\rho$ for $\rho>0$ small enough. Once this is true, using degree theory, we conclude that for some $|\xi|<\rho$ we have $G(\xi)=0$, which finishes the proof.

In order to prove this claim, note that for $\rho>0$ small, one has for all $|\xi|=\rho$,
\begin{equation*}
\langle \nabla w(\xi),\xi\rangle<0
\end{equation*}
Thus for $\delta>0$ small but fixed,
\begin{equation*}
\gamma:=\sup_{x\in B_\delta}\langle \nabla w(\xi),\xi\rangle <0 \mbox{ for \ all }|\xi|=\rho.
\end{equation*}
Then similarly to the proof of part (b.2) in Section 5 of \cite{ddmw}, one has
\begin{equation*}
\begin{split}
\Big|\lambda^{-2s}\int_{\R^n \setminus B_\delta}V\Big(\frac{x}{\lambda}\Big)u_\lambda(x+\xi)\langle \nabla w(x+\xi),\xi \rangle\Big|&\leq C\lambda^{n-2s},\\
\Big|\lambda^{-2s}\int_{B_{\lambda R}}V\Big(\frac{x}{\lambda}\Big)u_\lambda(x+\xi)\langle \nabla w(x+\xi),\xi \rangle\Big|&=O(\lambda^n).\\
\end{split}
\end{equation*}
Then
\begin{equation*}
\begin{split}
&\Big|\lambda^{-2s}\int_{B_\delta\setminus B_{\lambda R}}V(\frac{x}{\lambda})u_\lambda(x+\xi)\langle \nabla w(x+\xi),\xi \rangle\Big|\leq c\gamma\int_{B_\delta\setminus B_{\lambda R}}V\Big(\frac{x}{\lambda}\Big),\\
&\int_{B_\delta\setminus B_{\lambda R}}V\Big(\frac{x}{\lambda}\Big)
=\int_{B_\delta\setminus B_{\lambda R}}f\Big(\frac{x}{|x|}\Big)|x|^{-n}\,dx
+\int_{B_\delta\setminus B_{\lambda R}}|x|^{-n}\Big(V(x)|x|^n-f\Big(\frac{x}{|x|}\Big)\Big)\,dx,\\
&\int_{B_\delta\setminus B_{\lambda R}}f\Big(\frac{x}{|x|}\Big)|x|^{-n}\,dx=\log \frac{1}{\lambda}\int_{\mathbb S^{n-1}}f+O(1),\\
&\int_{B_\delta\setminus B_{\lambda R}}|x|^{-n}\Big(V(x)|x|^n-f\Big(\frac{x}{|x|}\Big)\Big)\,dx\leq \ve \log\frac{1}{\lambda},
\end{split}
\end{equation*}
for $\ve$ small enough if $R$ is large enough. Combining all the above estimates, one obtains that $\langle G(\xi),\xi\rangle <0$ for all $|\xi|=\rho$ small.\\
\end{proof}

\begin{remark}
For $p=\frac{n+2s-1}{n-2s-1}$, we can get the same solvability result as in case \emph{ii.} of Proposition \ref{invertibility} if one considers the new weighted norm
\begin{equation*}
\begin{split}
\|\phi\|_*&=\sup_{\{|x|\leq 1\}}|x|^\sigma |\phi(x)|+\sup_{\{|x|\geq 1\}}|x|^{\frac{2s}{p-1}+\alpha}|\phi(x)|,\\
\|h\|_{**}&=\sup_{\{|x|\leq 1\}}|x|^{\sigma+2s} |h(x)|+\sup_{\{|x|\geq 1\}}|x|^{\frac{2s}{p-1}+\alpha+2s}|h(x)|,
\end{split}
\end{equation*}
for some $\alpha>0$ small. In this new norm, one can also show that $w'$ is in the kernel space of $L^*$ even when $p=\frac{n+2s-1}{n-2s-1}$. With some minor modifications, we can also obtain existence of solutions.
\end{remark}

\bigskip

\noindent\textbf{Acknowledgements.}
 M.d.M. Gonz\'alez is supported by Spanish government grants MTM2014-52402-C3-1-P and MTM2017-85757-P, and the BBVA foundation grant for  Researchers and Cultural Creators, 2016. The research of H. Chan and J. Wei is supported by NSERC of Canada.

\end{document}